\documentclass[11pt,oneside]{amsart} 
\usepackage{amsmath, amssymb, verbatim, amscd, amsthm, mathrsfs, mathtools}
\usepackage[driverfallback=dvipdfm]{hyperref}
\usepackage{color, graphicx, shortvrb}
\usepackage{enumerate}
\usepackage[all]{xy}

\numberwithin{equation}{section}

\theoremstyle{definition} 
\newtheorem{defn}{Definition}[section]

\theoremstyle{plain}
\newtheorem{thm}[defn]{Theorem}
\newtheorem{lem}[defn]{Lemma}
\newtheorem{prop}[defn]{Proposition}

\theoremstyle{remark}
\newtheorem{rem}[defn]{Remark}

\usepackage[normalem]{ulem}
\usepackage{marginnote}
\usepackage[top=30truemm,bottom=25truemm,left=35truemm,right=35truemm]{geometry}

\newcommand{\C}{\mathbb C} 
\newcommand{\N}{\mathbb N}

\def\Q{\mathbb{Q}}

\newcommand{\set}[1]{\left\{ #1  \right\}}
\newcommand{\norm}[1]{\lVert #1 \rVert}

\author[I.~Matsuzaki]{Izuho Matsuzaki}
\address[I.~Matsuzaki]
{Graduate School of Science and Technology,
Niigata University, Niigata 950-2181, Japan}
\email{matsuzaki@m.sc.niigata-u.ac.jp}


\title[Ring isomorphisms in norm between unital $C^*$-algebras]{Involution-preserving ring isomorphisms in norm between unital \(C^*\)-algebras}

\begin{document}

\begin{abstract}
Let \(A\) and \(B\) be nonzero unital \(C^*\)-algebras with units
\(1_A\) and \(1_B\), respectively, and let
\(T\colon A\to B\) be a bijection satisfying
\[
\|T(a+b)\|=\|T(a)+T(b)\|,
\qquad
\|T(ab)\|=\|T(a)T(b)\|,
\qquad
T(a^*)=T(a)^*
\]
for all \(a,b\in A\).
We prove that there exist  a central symmetry \(u\) in \(B\) and a real \( * \)-isomorphism $\Phi\colon A\to B$ such that 
\[
T(a)=u\Phi(a)
\qquad(a\in A).
\]
Moreover, $u$ and $\Phi$ are uniquely determined by $T$.
Conversely, if \(u\in B\) is a central symmetry and
\(\Phi\colon A\to B\) is a real \( * \)-isomorphism, then  
\(T=u\Phi\) is a bijection satisfying the three identities above.
In particular, the two norm identities, together with involution
preservation, force the normalized map $uT$ to preserve the full product,
not merely the Jordan product.
\end{abstract}
\maketitle

\section{Introduction and main result}

A recurring question in Banach algebra theory is how weak  
conditions on sums and products determine
the algebraic structure of a map.
For instance, the Kowalski--Słodkowski theorem \cite{Kowalski_Slodkowski}
shows that a spectral condition on differences forces
a complex-valued functional \(\phi\) on a complex Banach algebra
satisfying \(\phi(0)=0\) to be both linear and multiplicative.
Molnár \cite{Molnar} obtained related characterizations of maps on operator and function algebras through spectral conditions on products.

In this paper, we investigate a related question concerning norms. 
More precisely, we study to what extent bijections between \(C^*\)-algebras preserve the algebraic structure when the identities defining ring isomorphisms
are weakened to norm identities.
Let \(A\) and \(B\) be  Banach algebras.
A bijection \(T\colon A\to B\) is called a
\emph{ring isomorphism in norm} if
\[
\|T(a+b)\|=\|T(a)+T(b)\|,
\qquad
\|T(ab)\|=\|T(a)T(b)\|
\]
for all \(a,b\in A\).
Observe that every ring isomorphism satisfies these conditions.
We prove that, for nonzero unital \(C^*\)-algebras, these norm identities together with preservation of the involution force the map to be a real \(*\)-isomorphism up to multiplication by a central symmetry. 
Neither linearity nor continuity is assumed.

Dong, Lin and Zheng \cite{generalization_of_G-K} introduced
this notion in their generalization of the
Gelfand--Kolmogoroff theorem \cite{GelfandKolmogoroff}.
Let \(X\) and \(Y\) be compact Hausdorff spaces, and let \(C(X,\mathbb K)\) denote the Banach algebra of all
continuous \(\mathbb K\)-valued functions on \(X\),
equipped with the supremum norm,
where \(\mathbb K=\mathbb R\) or \(\mathbb C\).
% The Gelfand--Kolmogoroff theorem \cite{GelfandKolmogoroff}
% states that every ring isomorphism between \(C(X,\mathbb R)\)
% and \(C(Y,\mathbb R)\) is a composition operator induced
% by a homeomorphism from \(Y\) onto \(X\).
Dong, Lin and Zheng showed that ring isomorphisms in norm between
\(C(X,\mathbb R)\)
and \(C(Y,\mathbb R)\) are weighted composition operators,
with weights taking values in \(\{\pm 1\}\).
Thus, after normalization by the weight, these maps
are real algebra isomorphisms.

For complex-valued functions, the two norm identities alone
do not ensure a weighted composition representation,
since even the complex field \(\mathbb C\) admits
discontinuous ring automorphisms \cite{Charnow,kestelman}.
Miura and Takahashi \cite{taira} considered ring isomorphisms
in norm between \(C(X,\mathbb C)\) and \(C(Y,\mathbb C)\)
that preserve complex conjugation.
They obtained a weighted composition representation,
allowing complex conjugation on a clopen subset of \(Y\).
The weight takes values in \(\{\pm 1\}\), and the normalized
map is a real \(*\)-isomorphism.

Since complex conjugation is the involution on
\(C(X,\mathbb C)\), it is natural to ask whether a similar
description holds for involution-preserving ring isomorphisms
in norm between unital \(C^*\)-algebras without assuming
commutativity.
We answer this question affirmatively.

To state the result, let \(A\) and \(B\) be unital
\(C^*\)-algebras, and let \(Z(B)\) denote the center of \(B\).
An element \(u\in Z(B)\) is called a \emph{central symmetry}
if \(u=u^*\) and \(u^2=1_B\), where \(1_B\) is the unit of \(B\).
A bijection \(\Phi\colon A\to B\) is called a
\emph{real \(*\)-isomorphism} if it is real-linear and satisfies
$\Phi(ab)=\Phi(a)\Phi(b)$ and $\Phi(a^*)=\Phi(a)^*$ for all $a,b\in A$.
Our main result is the following.

\begin{thm}\label{thm:main}
Let \(A\) and \(B\) be nonzero unital \(C^*\)-algebras with units
\(1_A\) and \(1_B\), respectively. Suppose that a bijection
\(T\colon A\to B\) satisfies
\begin{align}
  \|T(a+b)\|
&=\|T(a)+T(b)\|,
\label{eq:norm-additive}\\
\|T(ab)\|
&=\|T(a)T(b)\|,
\label{eq:norm-multiplicative}\\
T(a^*)&=T(a)^*
\label{eq:involution}
\end{align}
for every $a,b\in A$.
Then there exist a central symmetry \(u\in B\) and a real
\(*\)-isomorphism \(\Phi\colon A\to B\) such that
\[
T(a)=u\Phi(a)
\qquad
(a\in A).
\]
Moreover, \(u\) and \(\Phi\) are uniquely determined by \(T\).

Conversely, let \(u\in B\) be a central symmetry and let
\(\Phi\colon A\to B\) be a real $*$-isomorphism. 
Then the map \(T\colon A\to B\) given by
\(T(a)=u\Phi(a)\) for every \(a\in A\) is a bijection satisfying
\eqref{eq:norm-additive}, \eqref{eq:norm-multiplicative}, and
\eqref{eq:involution}.
\end{thm}

In the commutative case, Miura and Takahashi \cite{taira}
first show that the normalized map is a surjective
real-linear isometry. They then apply a Banach--Stone type
theorem for real-linear isometries to obtain its
representation, from which multiplicativity follows.
In the noncommutative case, however, establishing
the isometry property does not suffice to recover
multiplicativity.

Indeed, Kadison \cite{Kadison_isometry} proved that every
surjective complex-linear isometry between unital
\(C^*\)-algebras is the product of a unitary and a Jordan
\(*\)-isomorphism
(that is, a bijective complex-linear map preserving
squares and the involution).
Thus, even under complex linearity, removing the unitary
factor yields a map that need not preserve products.
In our setting, the product norm condition allows us
to establish multiplicativity of the normalized map.
The following example illustrates why the isometry
property alone is insufficient.

% The preservation of the associative product is important
% in the noncommutative setting.
% Kadison \cite{Kadison_isometry} proved that every surjective
% complex-linear isometry between unital \(C^*\)-algebras is
% the product of a unitary and a Jordan \(*\)-isomorphism
% (that is, a bijective complex-linear map preserving squares
% and the involution).
% In our setting, the normalized map $\Phi=uT$ is a real-linear isometry,
% and the product norm condition ensures that it also
% preserves the associative product.
% The following example shows why the isometry property
% alone is not sufficient, even for unital complex-linear maps.

\begin{rem}
For \(n\geq 2\), consider the transpose map
\(\tau\colon M_n(\mathbb{C})\to M_n(\mathbb{C})\), defined by
\[
\tau(a)=a^{\mathrm T}
\qquad
(a\in M_n(\mathbb{C})).
\]
This map is a unital surjective complex-linear isometry and a Jordan
\(*\)-automorphism, but it reverses the order of multiplication:
\[
\tau(ab)=\tau(b)\tau(a)
\qquad
(a,b\in M_n(\mathbb{C})).
\]
Moreover, it does not satisfy \eqref{eq:norm-multiplicative}. To see
this, let \(e_{ij}\) denote the standard matrix unit whose
\((i,j)\)-entry is \(1\) and whose other entries are \(0\). Taking
\(a=e_{12}\) and \(b=e_{11}\), we obtain
\(ab=0\) and \(\tau(a)\tau(b)=e_{21}\). Consequently,
\(\|\tau(ab)\|=0\), whereas
\(\|\tau(a)\tau(b)\|=1\). 
Thus, the product norm condition rules out anti-multiplicative
behavior that cannot be excluded by the isometric structure alone.
\end{rem}

We briefly outline the proof.
We first show that \(u=T(1_A)\) is a central symmetry and
normalize \(T\) by setting \(\Phi(a)=uT(a)\).
By analyzing the behavior of \(\Phi\) on self-adjoint
and positive elements, we prove that \(\Phi\) is
a surjective real-linear isometry.
We then apply \cite[Proposition~2.1]{Hatori_Watanae}.
The product norm condition implies that the complex-linear
and conjugate-linear parts preserve zero products.
Finally, \cite[Theorem~4.11]{Chebotar} shows that both parts
are multiplicative. This proves that \(\Phi\) is a real
\(*\)-isomorphism.

\section{Normalization of \texorpdfstring{\(T\)}{T}}

Throughout the remainder of the paper, let \(A\) and \(B\) be nonzero unital
\(C^*\)-algebras with units $1_A$ and $1_B$, respectively, and let \(T\colon A\to B\) be a bijection satisfying
\eqref{eq:norm-additive}, \eqref{eq:norm-multiplicative}, and
\eqref{eq:involution}.
Set
\[
u=T(1_A).
\]
In this section, we prove that \(u\) is a central symmetry
and use it to normalize \(T\), obtaining a map \(\Phi\)
with \(\Phi(1_A)=1_B\).
We then establish
the basic properties of \(\Phi\) needed to apply the relevant structure
theorems in the next section.

Applying Tabor's result \cite[Corollary~1]{tabor} with
\(G=(A,+)\), \(X=B\), and \(f=T\), we conclude from the
surjectivity of \(T\) and \eqref{eq:norm-additive} that \(T\) is
additive. Consequently,
\[
T(0)=0,
\qquad
T(ra)=rT(a)
\qquad
(r\in\mathbb{Q},\ a\in A).
\]

We first show that \(u\) is a symmetry.
For an element \(a\) of a unital $C^*$-algebra \(C\),
we write
\[
\sigma(a)=\{\lambda\in\mathbb C:
a-\lambda 1_C \text{ is not invertible in } C\}
\]
for the spectrum of \(a\).

\begin{lem}\label{lem:symmetry}
The element \(u\) is a symmetry; that is,
\[
u^*=u,
\qquad
u^2=1_B.
\]
\end{lem}

\begin{proof}
We first prove that \(u^*=u\).
It follows from \eqref{eq:involution} that 
\[
u^*
=T(1_A)^*
=T(1_A^*)
=T(1_A)
=u.
\]

To prove $u^2=1_B$,
we next show that \(\sigma(u)\subset [-1,1]\).
Let \(x\in B\). Since \(T\) is surjective, there exists \(a\in A\)
such that \(T(a)=x\). By \eqref{eq:norm-multiplicative},
\begin{equation}\label{ux_equal_x}
\norm{ux}
=\norm{T(1_A)T(a)}
=\norm{T(1_Aa)}
=\norm{T(a)}
=\norm{x}.
\end{equation}
Thus left multiplication by \(u\) is norm-preserving.
In particular,
taking \(x=1_B\) in \eqref{ux_equal_x}, we obtain
\[
\norm{u}
=\norm{u1_B}
=\norm{1_B}
=1.
\]
Since \(u=u^*\) and \(\norm{u}=1\), it follows that
$\sigma(u)\subset [-1,1]$.

It remains to show that
$\sigma(u)\subset\set{\pm1}$.
Suppose, to the contrary, that there exists
\(t_0\in\sigma(u)\) such that \(\lvert t_0\rvert<1\). Set
\[
V=\set{t\in\sigma(u):\lvert t\rvert<1}.
\]
Then \(V\) is an open neighborhood of \(t_0\) in \(\sigma(u)\).

By the continuous functional calculus, there exists an isometric
\( * \)-isomorphism
\[
\Gamma\colon C^*(u,1_B)\to C(\sigma(u),\C)
\]
such that \(\Gamma(u)=\mathrm{id}\).
Here \(C^*(u,1_B)\) denotes the unital \(C^*\)-subalgebra
of \(B\) generated by \(u\), and \(\mathrm{id}\) denotes
the identity function on \(\sigma(u)\).
We write \(\|\cdot\|_\infty\) for the supremum norm
on \(C(\sigma(u),\mathbb C)\).
It follows from
\eqref{ux_equal_x} that
\begin{equation}\label{eq:multiplication_norm_preserving}
\norm{\mathrm{id}\,\Gamma(y)}_\infty
=\norm{\Gamma(uy)}_\infty
=\norm{uy}
=\norm{y}
=\norm{\Gamma(y)}_\infty
\end{equation}
for every \(y\in C^*(u,1_B)\).

By Urysohn's lemma, there exists \(f_0\in C(\sigma(u),\C)\) such that
\[
 f_0(\sigma(u))\subset [0,1],
\qquad
f_0(t_0)=1=\norm{f_0}_{\infty},
\qquad
f_0(t)=0
\quad
(t\in\sigma(u)\setminus V).
\]
If \(t\in V\), then
$\lvert \mathrm{id}(t)f_0(t)\rvert
=\lvert tf_0(t)\rvert
\leq \lvert t\rvert
<1$,
whereas if \(t\in\sigma(u)\setminus V\), then
$\lvert \mathrm{id}(t)f_0(t)\rvert
=\lvert tf_0(t)\rvert
=0$.
Thus
\[
\lvert \mathrm{id}(t)f_0(t)\rvert<1
\qquad
(t\in\sigma(u)).
\]
Since the continuous function
\(\lvert \mathrm{id}\,f_0\rvert\) attains its maximum on the compact
set \(\sigma(u)\), it follows that
\[
\norm{\mathrm{id}\,f_0}_\infty<1.
\]

On the other hand, since \(\Gamma\) is surjective, there exists
\(y_0\in C^*(u,1_B)\) such that \(\Gamma(y_0)=f_0\). Applying
\eqref{eq:multiplication_norm_preserving} with \(y=y_0\), we obtain
\[
\begin{aligned}
\norm{\mathrm{id}\,f_0}_\infty
=\norm{\mathrm{id}\,\Gamma(y_0)}_\infty=\norm{\Gamma(y_0)}_\infty
=\norm{f_0}_\infty
=1,
\end{aligned}
\]
which is a contradiction. Consequently, $\sigma(u)\subset\set{\pm1}$.
The continuous functional calculus now gives \(u^2=1_B\). Together
with \(u^*=u\), this proves that \(u\) is a symmetry.
\end{proof}

Having established that \(u\) is a symmetry, we now prove its
centrality. 
Recall that two projections \(p\) and \(q\) are
said to be \textit{orthogonal} if \(pq=qp=0\).
\begin{lem}
The element $u$ belongs to $Z(B)$.
\end{lem}
\begin{proof}
Set
\[
p=\frac{1_B+u}{2},
\qquad
q=\frac{1_B-u}{2}.
\]
Since \(u\) is a symmetry, \(p\) and \(q\) are orthogonal
projections satisfying
\[
p+q=1_B,
\qquad
u=p-q.
\]

We show that \(pBq=\set{0}=qBp\).
First, we prove $pBq=\set{0}$.
Let \(x\in pBq\). 
Then 
\[x=pxq,\] and since \(p\) and \(q\) are orthogonal, we have \(x^2=0\).
Since $T$ is surjective, there exists $a\in A$ such that $T(a)=x$.
By \eqref{eq:norm-multiplicative},
\[
\norm{T(a^2)}
=\norm{T(a)^2}
=\norm{x^2}
=0.
\]
Thus \(T(a^2)=0\). Since \(T\) is injective and $T(0)=0$, it follows that
\(a^2=0\). 
Therefore, $(1_A+a)(1_A-a)=1_A$.
Moreover, since $u$ is a symmetry, we have $\norm{u}=1$. 
Hence, by \eqref{eq:norm-multiplicative}, we obtain  
\begin{align}\label{eq:2.2.1}
  \norm{T(1_A+a)T(1_A-a)}=\norm{T\big((1_A+a)(1_A-a)\big)}
  =\norm{T(1_A)}
  =\norm{u}=1.
\end{align}
On the other hand, since \(p\) and \(q\) are orthogonal projections 
and \(u=p-q\), we have
\[
ux=(p-q)pxq=x,
\qquad
xu=pxq(p-q)=-x.
\]
Hence  
$(u+x)(u-x)
=1_B-2x$ by Lemma~\ref{lem:symmetry}.
The additivity of \(T\) gives
$u+x=T(1_A+a)$ and 
$u-x=T(1_A-a)$.
The last three equalities, combined with \eqref{eq:2.2.1}, yield
\begin{align*}
\norm{1_B-2x}
=\norm{(u+x)(u-x)}=\norm{T(1_A+a)T(1_A-a)}
=1.
\end{align*}
Replacing \(x\) by \(-x\) in the preceding argument gives
\(\norm{1_B+2x}=1\).
For each choice of sign, the \(C^*\)-identity gives
\[
\norm{(1_B\pm2x)(1_B\pm2x)^*}
=\norm{1_B\pm2x}^2
=1.
\]
Since \((1_B\pm2x)(1_B\pm2x)^*\) is positive, we have $(1_B\pm2x)(1_B\pm2x)^*\leq1_B$.
Summing the inequalities corresponding to the two choices of sign
gives
\begin{align*}
2\cdot1_B
&\geq
(1_B+2x)(1_B+2x)^*
+(1_B-2x)(1_B-2x)^*\\
&=2\cdot1_B+8xx^*.
\end{align*}
Hence $0\geq xx^*$, and thus $x=0$.
Therefore,
$pBq=\set{0}$.
Taking adjoints gives
$qBp=(pBq)^*=\set{0}$.

To prove $u\in Z(B)$, let \(y\in B\). 
Since \(p+q=1_B\) and \(pBq=qBp=\{0\}\), we have
\[
py=py(p+q)=pyp,\qquad
yp=(p+q)yp=pyp.
\]
Thus \(py=yp\). Since \(y\in B\) was arbitrary, \(p\in Z(B)\).
Consequently, we conclude from $p+q=1_B$ that  
\[
u=p-q=2p-1_B\in Z(B).
\]
\end{proof}

Having shown that \(u\) is a central symmetry, we now use it to
normalize \(T\). 
Define a map \(\Phi\colon A\to B\) by
\[
\Phi(a)=uT(a)
\qquad (a\in A).
\]
Since \(u^2=1_B\), left multiplication by \(u\) is a bijection on
\(B\).
Because \(T\) is bijective, \(\Phi\) is also bijective.

We now verify the basic properties of this normalization.

\begin{prop}\label{prop:properties_of_Phi}
The bijection \(\Phi\) satisfies
\begin{align}
\Phi(a+b)&=\Phi(a)+\Phi(b),\label{eq:Phi_additive}\\
\norm{\Phi(ab)}&=\norm{\Phi(a)\Phi(b)},\label{eq:Phi_multiplicative}\\
\Phi(a^*)&=\Phi(a)^*,\label{eq:Phi_involution}\\
\Phi(1_A)&=1_B\label{eq:Phi_unit}
\end{align}
for all \(a,b\in A\).
Moreover, $\Phi$ preserves zero products; that is,
\begin{equation}\label{eq:Phi_zero_product}
ab=0
\quad\Longrightarrow\quad
\Phi(a)\Phi(b)=0
\end{equation}
for every $a,b\in A$.
\end{prop}

\begin{proof}

We first prove \eqref{eq:Phi_additive} and
\eqref{eq:Phi_multiplicative}. 
Let \(a,b\in A\). 
The additivity of $T$ implies that $\Phi$ is additive.
Since \(u\) is a central symmetry,
\[
\Phi(a)\Phi(b)
=uT(a)uT(b)
=u^2T(a)T(b)
=T(a)T(b).
\]
By \eqref{ux_equal_x} and \eqref{eq:norm-multiplicative},
we have
\[
\begin{aligned}
\norm{\Phi(ab)}
=\norm{uT(ab)}
 =\norm{T(ab)}=\norm{T(a)T(b)}
 =\norm{\Phi(a)\Phi(b)}.
\end{aligned}
\]
Therefore, we obtain \eqref{eq:Phi_additive} and \eqref{eq:Phi_multiplicative}.

We next prove \eqref{eq:Phi_involution} and \eqref{eq:Phi_unit}.
Let \(a\in A\). By \eqref{eq:involution} and the fact that \(u\) is a
central symmetry, we obtain
% \[
% \begin{aligned}
% \Phi(a)^*
% =(uT(a))^*
% =T(a^*)u
%  =uT(a^*)
%  =\Phi(a^*).
% \end{aligned}
% \]
\[
\Phi(a^*)=uT(a^*)=T(a^*)u=(uT(a))^*=\Phi(a)^*.
\]
Since \(u=T(1_A)\) and \(u^2=1_B\), we have
\[
\Phi(1_A)=uT(1_A)=u^2=1_B.
\]
Hence \eqref{eq:Phi_involution}  and \eqref{eq:Phi_unit} hold.

It remains to prove that $\Phi$ preserves zero products.
Let \(a,b\in A\) satisfy \(ab=0\). 
Since \(\Phi(0)=0\) by \eqref{eq:Phi_additive},
it follows from \eqref{eq:Phi_multiplicative} that
\[
\norm{\Phi(a)\Phi(b)}
=\norm{\Phi(ab)}
=\norm{\Phi(0)}
=0.
\]
Hence \(\Phi(a)\Phi(b)=0\). 
Therefore, \eqref{eq:Phi_zero_product} holds.
\end{proof}

\section{Proof of the main theorem}

In this section, we complete the proof of
Theorem~\ref{thm:main}.
We first show that the normalized map \(\Phi\) is a
surjective real-linear isometry, and then use its
zero-product-preserving property to prove multiplicativity.

We write \(A_{\mathrm{sa}}\) and \(B_{\mathrm{sa}}\) for the
self-adjoint parts of \(A\) and \(B\), respectively, and \(A_+\) and
\(B_+\) for their positive cones.
Since \(\Phi\) is bijective and preserves the involution, it maps the
self-adjoint part of \(A\) onto that of \(B\); that is,
\begin{equation}\label{eq:self-adjoint_preserve}
\Phi(A_{\mathrm{sa}})=B_{\mathrm{sa}}.
\end{equation}

We next establish a further property of $\Phi$.
\begin{lem}\label{lem:positive_preserve}
The map $\Phi$ preserves positivity; that is, $\Phi(A_+)\subset B_+$.
\end{lem}

\begin{proof}
Let \(a\in A_+\) and set \(b=a^{1/2}\in A_+\). Choose \(k\in\N\)
such that
\[
\norm{\Phi(b)}\leq k,
\]
and set 
\[c=\frac{b}{k}.\] 
Since $\Phi$ is additive, it is \(\Q\)-linear, and hence
$\Phi(c)=\frac{1}{k}\Phi(b)$.
Thus $\norm{\Phi(c)}\leq 1$.
Moreover, since \(c\in A_{\mathrm{sa}}\), it follows from
\eqref{eq:self-adjoint_preserve} that \(\Phi(c)\in B_{\mathrm{sa}}\). 
Therefore,
$-1_B\leq \Phi(c)\leq 1_B$,
and hence
$0\leq \Phi(c)^2\leq 1_B$.
Consequently,
\[
\norm{1_B-\Phi(c)^2}\leq 1.
\]

Using \eqref{eq:Phi_unit}, \eqref{eq:Phi_additive}, and
\eqref{eq:Phi_multiplicative}, we have
\begin{align*}
\norm{1_B-\Phi(c^2)}=\norm{\Phi(1_A-c^2)}
=\norm{\Phi\bigl((1_A-c)(1_A+c)\bigr)}
=\norm{\Phi(1_A-c)\Phi(1_A+c)}.
\end{align*}
By additivity, \eqref{eq:Phi_unit}, and the estimate above, we obtain
\begin{align*}
\norm{\Phi(1_A-c)\Phi(1_A+c)}&=\norm{(1_B-\Phi(c))(1_B+\Phi(c))}\\&=\norm{1_B-\Phi(c)^2}
\leq 1.
\end{align*}
Combining these relations gives
\(\norm{1_B-\Phi(c^2)}\leq1\).
Since \(c^2\in A_{\mathrm{sa}}\), we have
\(\Phi(c^2)\in B_{\mathrm{sa}}\). 
The preceding inequality therefore
implies that
$\sigma(\Phi(c^2))\subset [0,2]$,
so that $\Phi(c^2)\geq 0$.

Finally, since \(a=b^2=k^2c^2\), the additivity of $\Phi$ gives
\[
\Phi(a)=k^2\Phi(c^2)\geq0.
\]
Thus $\Phi(A_+)\subset B_+$.
\end{proof}

The next lemma shows that $\Phi$ does not increase norms on the
self-adjoint part of $A$.

\begin{lem}\label{lem:bounded}
For every $a\in A_{\mathrm{sa}}$, we have
$\norm{\Phi(a)}\leq\norm{a}$.
\end{lem}

\begin{proof}
We first note that $\Phi$ is order-preserving on
$A_{\mathrm{sa}}$. 
Indeed, if $a,b\in A_{\mathrm{sa}}$ satisfy
$a\leq b$, then $b-a\in A_+$. 
Hence, by the additivity of $\Phi$ and Lemma~\ref{lem:positive_preserve},
\[
\Phi(b)-\Phi(a)=\Phi(b-a)\in B_+,
\]
so that $\Phi(a)\leq\Phi(b)$.

Now let $a\in A_{\mathrm{sa}}$, and choose $r\in\mathbb{Q}$ such that
$r>\norm{a}$. Since $a$ is self-adjoint,
\[
-r1_A\leq a\leq r1_A.
\]
The order preservation established above, together with the
$\mathbb{Q}$-linearity of $\Phi$ and \eqref{eq:Phi_unit}, gives
\[
-r1_B\leq\Phi(a)\leq r1_B.
\]
Since $\Phi(a)\in B_{\mathrm{sa}}$, it follows that
$\norm{\Phi(a)}\leq r$. As this holds for every
$r\in\mathbb{Q}$ with $r>\norm{a}$, we conclude that $\norm{\Phi(a)}\leq\norm{a}$.
\end{proof}

The following lemma upgrades the preceding norm estimate to equality.

\begin{lem}\label{lem:norm-preserving}
For every \(a\in A_{\mathrm{sa}}\), we have $\norm{\Phi(a)}=\norm{a}$.
\end{lem}

\begin{proof}

By \eqref{eq:Phi_additive} and \eqref{eq:self-adjoint_preserve}, the restriction
\[
\Phi|_{A_{\mathrm{sa}}}
\colon A_{\mathrm{sa}}\to B_{\mathrm{sa}}
\]
is an additive bijection. 
For notational simplicity, we denote this
restriction by $\Phi_{\mathrm{sa}}$.
For $a,b\in A_{\mathrm{sa}}$,
Lemma~\ref{lem:bounded} gives
\[
\norm{\Phi_{\mathrm{sa}}(a)-\Phi_{\mathrm{sa}}(b)}
=
\norm{\Phi_{\mathrm{sa}}(a-b)}
\leq\norm{a-b}.
\]
Thus $\Phi_{\mathrm{sa}}$ is continuous. Since it is additive, it is
real-linear. Consequently, $\Phi_{\mathrm{sa}}$ is a bounded
real-linear bijection between the real Banach spaces
$A_{\mathrm{sa}}$ and $B_{\mathrm{sa}}$.
The open mapping theorem therefore yields a constant
\(M>0\) such that
\[
\norm{a}
\leq M\norm{\Phi_{\mathrm{sa}}(a)}
=M\norm{\Phi(a)}
\qquad(a\in A_{\mathrm{sa}}).
\]

Let \(a\in A_{\mathrm{sa}}\). 
For every nonnegative integer \(j\), both \(a^{2^j}\) and \(\Phi(a^{2^j})\) are self-adjoint by \eqref{eq:self-adjoint_preserve}. 
Hence
\eqref{eq:Phi_multiplicative} and the \(C^*\)-identity give
\[
\begin{aligned}
\norm{\Phi(a^{2^{j+1}})}=\norm{\Phi(a^{2^j}a^{2^j})}=\norm{\Phi(a^{2^j})^2}=\norm{\Phi(a^{2^j})}^2.
\end{aligned}
\]
It follows by induction that
\[
\norm{\Phi(a^{2^n})}
=\norm{\Phi(a)}^{2^n}
\qquad(n\in\N).
\]
Similarly, since \(a\) is self-adjoint, we have $\norm{a^{2^n}}=\norm{a}^{2^n}$.
Consequently,
\[
\begin{aligned}
\norm{a}^{2^n}
=\norm{a^{2^n}}
\leq M\norm{\Phi(a^{2^n})}=M\norm{\Phi(a)}^{2^n}.
\end{aligned}
\]
Taking the \(2^n\)-th root of both sides, we obtain
\[
\norm{a}\leq M^{1/2^n}\norm{\Phi(a)}.
\]
Letting \(n\to\infty\) gives
\[
\norm{a}\leq\norm{\Phi(a)}.
\]
The reverse inequality follows from Lemma~\ref{lem:bounded}.
Therefore, $\norm{\Phi(a)}=\norm{a}$.
\end{proof}

We now extend norm preservation from $A_{\mathrm{sa}}$ to all of $A$.
Together with the additivity of $\Phi$, this yields the desired
real-linear isometric structure.
\begin{prop}\label{prop:real-linear_isometry}
The map $\Phi$ is a surjective real-linear isometry. 
\end{prop}

\begin{proof}
Let \(a\in A\). 
Since $aa^*\in A_{\mathrm{sa}}$,
Lemma~\ref{lem:norm-preserving} is applicable. Therefore, the
$C^*$-identity, \eqref{eq:Phi_involution},
\eqref{eq:Phi_multiplicative}, and
Lemma~\ref{lem:norm-preserving} give
\[
\norm{\Phi(a)}^2
=\norm{\Phi(a)\Phi(a^*)}
=\norm{\Phi(aa^*)}
=\norm{aa^*}
=\norm{a}^2.
\]
Thus $\norm{\Phi(a)}=\norm{a}$ for every $a\in A$. Since $\Phi$ is
additive, it follows that $\Phi$ is a real-linear isometry. 
% Its
% surjectivity follows from Proposition~\ref{prop:properties_of_Phi}.
Therefore, $\Phi$ is a surjective real-linear isometry.
\end{proof}

We use the following consequence of a result of Hatori and Watanabe.

\begin{lem}[{\cite[Proposition~2.1]{Hatori_Watanae}}]
\label{lem:Hatori-Watanabe}
Let \(C\) and \(D\) be unital \(C^*\)-algebras with units $1_C$ and $1_D$, respectively, and let
\(R\colon C\to D\) be a surjective real-linear isometry satisfying $R(1_C)=1_D$.
Then there exist a central projection \(e\in D\) and a complex-linear Jordan
\( * \)-isomorphism \(J\colon C\to D\) such that
\[
R(c)=eJ(c)+(1_D-e)J(c)^*
\qquad(c\in C).
\]
\end{lem}

By Proposition~\ref{prop:real-linear_isometry} and \eqref{eq:Phi_unit}, the map
$\Phi$ satisfies the hypotheses of Lemma~\ref{lem:Hatori-Watanabe}.
Hence there exist a central projection $e_1\in B$ and a complex-linear
Jordan $*$-isomorphism $J\colon A\to B$ such that
\[
\Phi(a)=e_1J(a)+e_2J(a)^*
\qquad
(a\in A),
\]
where  $e_2=1_B-e_1$.
Then $e_1$ and $e_2$ are orthogonal central projections satisfying
$e_1+e_2=1_B$. 
For $j=1,2$, the ideal $e_jB$ is a unital
$C^*$-algebra with unit $e_j$.

For $j=1,2$, define a map $\Phi_j\colon A\to e_jB$ by 
\[
\Phi_j(a)=e_j\Phi(a)
\qquad
(a\in A).
\]
Then $\Phi=\Phi_1+\Phi_2$, and the representation of $\Phi$ above
gives
\[
\Phi_1(a)=e_1J(a),
\qquad
\Phi_2(a)=e_2J(a)^*
\qquad
(a\in A).
\]
Since $J$ is complex-linear, $\Phi_1$ is complex-linear, whereas
$\Phi_2$ is conjugate-linear. 
Moreover, since $J(1_A)=1_B$,
\begin{equation}\label{eq:Phi_j_unital}
\Phi_1(1_A)=e_1,
\qquad
\Phi_2(1_A)=e_2.
\end{equation}

We next prove that $\Phi$ is multiplicative. Since $e_1B$ and $e_2B$
are orthogonal ideals and $\Phi=\Phi_1+\Phi_2$, it suffices to show
that both $\Phi_1$ and $\Phi_2$ are multiplicative. 
We begin by recording some properties of these maps.
\begin{lem}\label{lem:Phi1-Phi2-bounded-linear-unital}
  For each $j\in\set{1,2}$, the map $\Phi_j$ is surjective and bounded, and it
preserves zero products.
\end{lem}
\begin{proof}

Fix $j\in\{1,2\}$. We first prove that $\Phi_j$ is surjective. Let
$x\in e_jB$. Since $\Phi$ is surjective, there exists $a\in A$ such
that $\Phi(a)=x$. 
Because $e_j$ is the unit of $e_jB$, we have $e_jx=x$, and therefore
\[
\Phi_j(a)=e_j\Phi(a)=e_jx=x.
\]
Thus $\Phi_j$ is surjective.

We next prove that $\Phi_j$ is bounded. Since $\Phi$ is
norm-preserving and $\norm{e_j}\leq1$, we have
\[
\norm{\Phi_j(a)}
=\norm{e_j\Phi(a)}
\leq\norm{\Phi(a)}
=\norm{a}
\qquad
(a\in A).
\]
Thus $\Phi_j$ is bounded.

It remains to show that each \(\Phi_j\) preserves zero products.
Suppose that \(a,b\in A\) satisfy
\(ab=0\). 
By
\eqref{eq:Phi_zero_product}, we have $\Phi(a)\Phi(b)=0$. Since $e_j$ is
a central projection,
\[
\Phi_j(a)\Phi_j(b)
=e_j\Phi(a)e_j\Phi(b)
=e_j\Phi(a)\Phi(b)
=0.
\]
Thus $\Phi_j$ preserves zero products.
\end{proof}

We use the following immediate consequence of a theorem of Chebotar
et al.

\begin{lem}[{\cite[Theorem~4.11]{Chebotar}}]
\label{lem:multiplicative}
Let $C$ and $D$ be unital $C^*$-algebras with units $1_C$ and $1_D$, respectively, and let
$U\colon C\to D$ be a surjective bounded complex-linear map satisfying
$U(1_C)=1_D$. If $U$ preserves zero products, then $U$ is
multiplicative.
\end{lem}

We now apply Lemma~\ref{lem:multiplicative} to the two components of
$\Phi$. For the conjugate-linear component $\Phi_2$, we pass to the
opposite algebra.

\begin{lem}\label{lem:Phi1-Phi2-multiplicative}
The maps \(\Phi_1\) and \(\Phi_2\) are multiplicative; that is,
\[
\Phi_j(ab)=\Phi_j(a)\Phi_j(b)
\qquad(a,b\in A,\ j=1,2).
\]
\end{lem}

\begin{proof}

We first consider the complex-linear map $\Phi_1$. 
If \(e_1=0\), then \(\Phi_1=0\) is multiplicative.
Suppose that \(e_1\ne0\).
By
Lemma~\ref{lem:Phi1-Phi2-bounded-linear-unital} and
\eqref{eq:Phi_j_unital}, the map $\Phi_1\colon A\to e_1B$ satisfies
the hypotheses of Lemma~\ref{lem:multiplicative}. Hence $\Phi_1$ is
multiplicative.

We next consider the conjugate-linear map $\Phi_2$. 
If \(e_2=0\), then \(\Phi_2=0\) is multiplicative.
Suppose that \(e_2\ne0\).
Let
$(e_2B)^{\mathrm{op}}$ denote the opposite $C^*$-algebra of $e_2B$.
It is a unital $C^*$-algebra with unit $e_2$ and has the same
underlying linear space, norm, and involution as $e_2B$. 
Its product
is defined by
\[
x\mathbin{\circ_{\mathrm{op}}}y=yx
\qquad
(x,y\in e_2B).
\]

Define a map $
\Psi_2\colon A\to (e_2B)^{\mathrm{op}}$ by 
\[
\Psi_2(a)=\Phi_2(a)^*
\qquad(a\in A).
\]
Since both $\Phi_2$ and the involution are conjugate-linear, $\Psi_2$
is complex-linear. 
Moreover, since the involution is a surjective isometry, Lemma~\ref{lem:Phi1-Phi2-bounded-linear-unital} implies that
$\Psi_2$ is surjective and bounded. 
By
\eqref{eq:Phi_j_unital},
\[
\Psi_2(1_A)=\Phi_2(1_A)^*=e_2.
\]
Hence $\Psi_2$ is a surjective bounded complex-linear map satisfying $\Psi_2(1_A)=e_2$.

To apply Lemma~\ref{lem:multiplicative} to \(\Psi_2\), it remains to show that \(\Psi_2\) preserves zero products. 
Suppose that \(a,b\in A\) satisfy \(ab=0\).
Since \(\Phi_2\) preserves zero products, we have
\(\Phi_2(a)\Phi_2(b)=0\). Hence
\[
\Psi_2(a)\mathbin{\circ_{\mathrm{op}}}\Psi_2(b)
=\Phi_2(b)^*\Phi_2(a)^*
=(\Phi_2(a)\Phi_2(b))^*
=0.
\]

Thus $\Psi_2$ satisfies all the hypotheses of
Lemma~\ref{lem:multiplicative} and is therefore multiplicative.
Consequently, for $a,b\in A$,
\[
\begin{aligned}
\Phi_2(ab)
&=\Psi_2(ab)^*
 =\bigl(\Psi_2(a)\mathbin{\circ_{\mathrm{op}}}\Psi_2(b)\bigr)^*\\
&=(\Psi_2(b)\Psi_2(a))^*
 =\Phi_2(a)\Phi_2(b).
\end{aligned}
\]
Hence $\Phi_2$ is multiplicative.
\end{proof}

We are now ready to complete the proof of the main theorem. We begin
by combining the multiplicativity of the two components with the
orthogonality of their ranges.

\begin{proof}[\textbf{Proof of Theorem~\ref{thm:main}}]
Let \(a,b\in A\). 
Since \(e_1\) and \(e_2\) are  orthogonal
central projections, we have
\[
\begin{aligned}
\Phi_1(a)\Phi_2(b)
  &=e_1\Phi(a)e_2\Phi(b)=0,\\
\Phi_2(a)\Phi_1(b)
  &=e_2\Phi(a)e_1\Phi(b)=0.
\end{aligned}
\]
Therefore, Lemma~\ref{lem:Phi1-Phi2-multiplicative}, combined with $\Phi=\Phi_1+\Phi_2$, gives
\[
\begin{aligned}
\Phi(a)\Phi(b)
&=(\Phi_1(a)+\Phi_2(a))(\Phi_1(b)+\Phi_2(b))\\
&=\Phi_1(a)\Phi_1(b)+\Phi_2(a)\Phi_2(b)\\
&=\Phi_1(ab)+\Phi_2(ab)
=\Phi(ab).
\end{aligned}
\]
Thus $\Phi$ is multiplicative. Together with
\eqref{eq:Phi_involution} and
Proposition~\ref{prop:real-linear_isometry}, this shows that $\Phi$ is a
real $*$-isomorphism.
Moreover, 
since $\Phi(a)=uT(a)$ and $u^2=1_B$, we obtain the desired representation
\[
T(a)=u\Phi(a)
\qquad
(a\in A).
\]

We next prove uniqueness. Suppose that
\[
T(a)=v\Psi(a)
\qquad (a\in A),
\]
where \(v\) is a central symmetry in \(B\) and
\(\Psi\colon A\to B\) is a real \({*}\)-isomorphism.
Since the real $*$-isomorphism $\Psi$ satisfies $\Psi(1_A)=1_B$, we have
\[
u=T(1_A)=v\Psi(1_A)=v.
\]
Consequently,
\[
\Phi(a)=uT(a)=uv\Psi(a)=\Psi(a)
\qquad
(a\in A),
\]
where we have used $u=v$ and $u^2=1_B$. Thus $v=u$ and
$\Psi=\Phi$, proving uniqueness.

Conversely, let $u$ be a central symmetry in $B$, let
$\Phi\colon A\to B$ be a real $*$-isomorphism, and define
\[
T(a)=u\Phi(a)
\qquad
(a\in A).
\]
Since left multiplication by the unitary element $u$ is bijective and
$\Phi$ is bijective, $T$ is bijective. 
The additivity of $\Phi$
implies that $T$ is additive, and hence \eqref{eq:norm-additive} holds.

Let $a,b\in A$. 
Since $u$ is a central symmetry and $\Phi$ is
multiplicative,
\[
T(a)T(b)
=u\Phi(a)u\Phi(b)
=\Phi(a)\Phi(b)
=\Phi(ab).
\]
Because $u$ is unitary, it follows that
\[
\norm{T(ab)}
=\norm{u\Phi(ab)}
=\norm{\Phi(ab)}
=\norm{T(a)T(b)}.
\]
Thus \eqref{eq:norm-multiplicative} holds. Finally, since $u$ is a
central symmetry and $\Phi$ preserves the involution,
\[
T(a^*)
=u\Phi(a^*)
=u\Phi(a)^*
=(u\Phi(a))^*
=T(a)^*.
\]
Hence \eqref{eq:involution} also holds. This completes the proof.
\end{proof}

\begin{rem}
By \cite[Proposition~3.9]{Tomforde}, the real \( * \)-isomorphism
\(\Phi\) decomposes into a complex-linear part and a conjugate-linear
part over complementary central projections. Since a real
\( * \)-isomorphism carries central projections bijectively onto
central projections, if \(A\) is simple, then one of these two
summands vanishes. Hence \(\Phi\), and therefore \(T\), is either
complex-linear or conjugate-linear.
\end{rem}

\subsection*{Acknowledgments}
The author is very grateful to Professor Takeshi Miura
for his valuable comments and helpful discussions.
The author was supported by JST SPRING,
Grant Number \mbox{JPMJSP2121}.


\begin{thebibliography}{99}

\bibitem{Charnow}
    A. Charnow, \textit{The automorphisms of an algebraically closed field},
Canad. Math. Bull. \textbf{13} (1970), 95--97.

  \bibitem{Chebotar}
  M.~A. Chebotar, W.-F. Ke, P.-H. Lee and N.-C. Wong,
  \textit{Mappings preserving zero products},
  Studia Math. \textbf{155} (2003), no.~1, 77--94.

   \bibitem{generalization_of_G-K}
  Y. Dong, P.-K. Lin and B. Zheng,
  \textit{A generalization of the Gelfand--Kolmogoroff theorem},
  Publ. Math. Debrecen \textbf{94} (2019), no. 1--2, 263--268.

  \bibitem{GelfandKolmogoroff}
  I. Gelfand and A. Kolmogoroff, \textit{On rings of continuous functions on topological spaces},
  Dokl. Akad. Nauk. SSSR \textbf{22} (1939), 11--15.

  \bibitem{Hatori_Watanae}
  O. Hatori and K. Watanabe, 
  \textit{Isometries between groups of invertible elements in $C^*$-algebras}, Studia Math. \textbf{209} (2012), no. 2, 103--106.


  \bibitem{Kadison_isometry}
  R.~V. Kadison, 
  \textit{Isometries of operator algebras}, Ann. of Math. (2) \textbf{54} (1951), 325--338.

  \bibitem{kestelman}
    H. Kestelman, \textit{Automorphisms of the field of complex numbers},
Proc. London Math. Soc. (2) \textbf{53} (1951), 1--12.
  
  \bibitem{Kowalski_Slodkowski}
  S. Kowalski and Z. Słodkowski, \textit{A characterization of multiplicative linear functionals in Banach algebras}, Studia Math. \textbf{67} (1980), no. 3, 215--223.


 

  


  

  \bibitem{taira}
  T.~Miura and T.~Takahashi,
  \textit{Ring isomorphisms in norm between
  Banach algebras of continuous complex-valued functions},
  preprint, arXiv:2601.11165v1 (2026).

  \bibitem{Molnar}
  L. Molnár, \textit{Some characterizations of the automorphisms of $B(H)$ and $C(X)$}, Proc. Amer. Math. Soc. \textbf{130} (2002), no. 1, 111--120.


  \bibitem{tabor}
  J. Tabor, \textit{Stability of the Fischer--Musz\'ely functional equation}, Publ. Math. Debrecen \textbf{62} (2003), no. 1--2, 205--211.


  \bibitem{Tomforde}
  M.~Tomforde,
  \textit{Continuity of ring \( * \)-homomorphisms between
  \(C^*\)-algebras},
  New York J. Math. \textbf{15} (2009), 161--167.

\end{thebibliography}
\end{document}